\documentclass[12pt]{amsart} %
\usepackage{eurosym, color}

\usepackage{DKstyle}

\newcommand{\vertiii}[1]{{\left\vert\kern-0.25ex\left\vert\kern-0.25ex\left\vert #1
    \right\vert\kern-0.25ex\right\vert\kern-0.25ex\right\vert}}
\theoremstyle{plain}
\makeatletter
\newcommand*{\rom}[1]{\expandafter\@slowromancap\romannumeral #1@}
\makeatother

\newtheorem*{thm*}{Theorem}

\usepackage{caption}
\usepackage[labelfont=rm]{subcaption}
\usepackage{bm}
\usepackage[pagebackref=true, colorlinks]{hyperref}

\hypersetup{pdffitwindow=true,linkcolor=blue,citecolor=blue,urlcolor=blue,menucolor=blue}

\usepackage{comment}


\date{\today}%
\dedicatory{}%
\commby{}%

\title{Norm bounds on Eisenstein series.}
\author{Dubi Kelmer}
\thanks{Kelmer is partially supported by NSF CAREER grant DMS-1651563.}
\email{kelmer@bc.edu}
\address{Department of Mathematics, Boston College, Boston, Massachusetts, United States}
\author{Alex Kontorovich}
\thanks{Kontorovich is partially supported by NSF grant DMS-1802119 and BSF grant 2020119.}
\email{alex.kontorovich@rutgers.edu}
\address{Department of Mathematics, Rutgers University, New Brunswick, New Jersey, United States}
\author{Christopher Lutsko}
\email{chris.lutsko@rutgers.edu}
\address{Department of Mathematics, Rutgers University, New Brunswick, New Jersey, United States} 
\date{July 2023}

\begin{document}

\maketitle
\begin{abstract}
\noindent    
We study the sup-norm bound (both individually and on average) for Eisenstein series on certain arithmetic hyperbolic orbifolds producing sharp exponents for the modular surface and Picard 3-fold. The methods involve bounds for Epstein zeta functions, and counting restricted values of indefinite quadratic forms at integer points.
\end{abstract}













\section{Introduction}

For a compact Riemannian manifold, $X$, one can show that if $\phi\in L^2(X)$ has $\|\phi\|_2=1$ and 
is an eigenfunction of the Laplace-Beltrami operator with
eigenvalue $\lambda$, then $\|\phi\|_{\infty}\ll \lambda^{\frac{\dim(X)-1}{4} }$, see \cite[Cor 2.2]{SeegerSogge1989}. This bound, usually referred to as the convexity bound, is sharp in general. However, when $X$ has negative curvature, it is believed that this exponent can be improved (a log savings is obtained by B\'erard \cite{Berard1977}) and there are some results of this nature for cusp forms on some arithmetic hyperbolic manifolds. Explicitly, for arithmetic hyperbolic surfaces it is conjectured that $\|\phi\|_\infty\ll_\epsilon\lambda^\epsilon$, and it was shown in  \cite{IwaniecSarnak1995} that $\|\phi\|_\infty\ll_\epsilon \lambda^{5/24+\epsilon}$  for $\phi$ a Hecke-Maass cusp form. 
In higher dimensions, the situation is more complicated, as it was shown in the work Rudnick and Sarnak \cite{RudnickSarnak1994} and in more detail by Mili\'{c}evi\'{c} \cite{Milicevic2011}, that for any $
\epsilon>0$, there exists Hecke-Maass forms on a given arithmetic hyperbolic $3$-manifold for which $\|\phi\|_\infty\gg \lambda^{1/4-\e}$.  Nevertheless, a subconvex upper bound of order  $\lambda^{5/12+\epsilon}$ was proved in \cite{Koyama1995, BlomerHarcosMilicevic2016}, so the 
truth is somewhere in between.


This paper is concerned with the analogous problem where the cusp form is replaced by an Eisenstein series. Explicitly, given a non-uniform lattice, $\Gamma$, acting on hyperbolic $n+1$ space $\bH^{n+1}$, for each cusp $\xi$, let 
$E_{\Gamma,\xi}(s,z)$ denote the Eisenstein series corresponding to this cusp with normalization such that the constant term of the Fourier expansion based at $\xi$ is of the form $y^s + c_\xi(s) y^{n-s}$. Then $E_{\Gamma,\xi}(\tfrac{n}{2}+it,z)$ is an almost-$L^2$ eigenfunction of the Laplacian with eigenvalue $\lambda=\tfrac{n^2}{4}+t^2$. Since the Eisenstein series is unbounded as $z$ moves into the cusp, in order to consider the supremum norm, we need to restrict to a compact set. We define the parameter
$\nu_\infty=\nu_\infty(\Gamma)$ as the infimum over all $\nu>0$ such that for any compact set $\Omega\subseteq \bH^{n+1}$,  and any cusp $\xi$,  we have 
$$\sup_{z\in\Omega}\ |E_{\Gamma,\xi}(\tfrac{n}{2}+it,z)|\ \ll_\Omega \  |t|^\nu.$$

\begin{rem}\label{rmk:Enormalize}
    There is another natural way to normalize the Eisenstein series; instead of ensuring that the constant term is of the form $y^s+c_\xi(s)y^{n-s}$, one could $L^2$-normalize (locally in $\Omega$, since Eisenstein series are not $L^2$), that is,  ask that 
    $$
    \int_{\Omega}\left|E_{\G,\xi}\left(
    \frac n2 + it, z\right)\right|^2  dz= 1,
    $$
where $dz$ is the hyperbolic volume form.
When $\Gamma$ is arithmetic, the two normalizations are not too different, but in the nonarithmetic setting, very little is known about the arising discrepancies.
\end{rem}

For numerous applications,
it suffices to understand sup norm bounds on average;
towards this, we
consider 
the 
quantity
$\nu_2=\nu_2(\Gamma)$ defined as the infimum of all $\nu>0$ such that for 
any compact set $\Omega\subseteq \mathbb{H}^{n+1}$, and any cusp $\xi$, we have
$$
\sup_{z\in\Omega}\int_{-T}^{T}
|E_{\Gamma,\xi}(\tfrac{n}{2}+it,z)|^2dt \ \ll_\Omega \ T^{1+2\nu}.$$
Then clearly $\nu_2\leq \nu_\infty$ but one expects that we can give a sharper bound for $\nu_2$.

\begin{rem}\label{rmk:appl}
    One example of an application is as follows. Given a rational quadratic form $Q$ of signature $(n+1,1)$,  the number $N(X)$ of primitive integer points $v\in\Z^{n+2}$ on the light cone $Q=0$ of norm bounded by $X$ can be estimated precisely in terms of $\nu_2=\nu_2(\SO_Q(\Z))$, see \cite[Prop 3.5]{KelmerYu2022} showing that 
    $$
    N(X) = c X^n + O(X^{n\left(1-\frac 1{2(\nu_2+1)}\right)}).
    $$
For another application, see \cite{BurrinNevoRuhrWeiss2020}.    
\end{rem}


\subsection{The case $\G=\SL_2(\Z)$}\

The main result of this paper is a determination of $\nu_2$ for the modular group. 
\begin{thm}\label{thm:1}
For $\Gamma=\SL_2(\Z)$, 
we have that $\nu_2(\Gamma)=0$.
That is, for any compact set $\Omega\subset\G\bk \bH$, there is a constant $c=c(\Omega)$ such that for  all $z\in\Omega$ and all $T\geq 1$,
\begin{eqnarray}\label{eq:thm1a}
    \int_{-T}^{T} |E_{\G}(z,\tfrac{1}{2}+it)|^2dt \ \leq \ cT\log^4(T). 
\end{eqnarray}
Here $E_\G$ is the Eisenstein series at the (unique) cusp at $\infty$.
\end{thm}

\begin{remark}
    For $\Gamma\leq \SL_2(\Z)$ a congruence subgroup, it is believed that $\nu_2(\Gamma)=\nu_\infty(\Gamma)=0$.    
    We expect that modifications of our techniques would show that $\nu_2(\Gamma)=0$ also for congruence subgroups of $\SL_2(\Z)$.
    For $\nu_\infty$, the convexity bound is $\nu_\infty(\Gamma)\leq \tfrac{1}{2}$, and the work of Young and consequently Huang  \cite{Young2018, Huang2019} using amplification gives the sub-convex bound of $\nu_\infty(\Gamma)\leq \tfrac{3}{8}$ while the best result to date is due to Blomer \cite{Blomer2020} who used the approximate functional equation and Burgess' bound to prove that $\nu_\infty(\SL_2(\Z))\leq \tfrac{1}{3}$. 
\end{remark}


\begin{rem}\label{rmk:PhilSar}
We note that for a general non-arithmetic lattice acting on $\bH^2$, 
nothing is known about $\nu_\infty(\Gamma)$.
For $\nu_2(\Gamma)$,
it is likely that the known convexity bound $\nu_2(\Gamma)\leq \tfrac{1}{2}$ is actually sharp (with the caveat that our normalization is not $L^2$, but rather that which arises in the pre-trace formula; see Remark \ref{rmk:Enormalize}).  As evidence for this, we show that if $\nu_2(\Gamma,z)<\tfrac{1}{2}$ at some point $z\in \bH^2$, then $\Gamma$ must have infinitely many cusp forms, in contrast with the Phillips-Sarnak conjecture \cite{PhillipsSarnak1985}, see Theorem \ref{thm:PhilSar} below and the discussion in Section \ref{sec:PhilSar}.
(Here we denote by $\nu_2(\Gamma,z)$ the exponent for a fixed point $z$.)
Again,
this might not necessarily signal the existence of  a large 
jump in the size of the Eisenstein series near the single point $z$, but could rather come from the discrepancy with $L^2$ normalization.
\end{rem}

The bulk of the paper is devoted to proving Theorem \ref{thm:1}. A key ingredient of independent interest is obtaining strong bounds on the Epstein zeta function  (see \S\ref{sec:Epstein}).

\subsection{The case $\G=\SL_2(\Z[i])$} \

For the Picard group $\G=\SL_2(\Z[i])$
acting on hyperbolic 3-space, one can resolve the issue completely. 
Using
Blomer's upper bounds \cite{Blomer2020} 
on Epstein zeta function,
and connections between such and Eisenstein series 
(which are well-known for $\SL_2(\Z)$ and derived 
here for $\Gamma = \SL_2(\Z[i])$),
 one obtains
the following. 
 \begin{thm}\label{thm:1b}
For $\Gamma=\SL_2(\Z[i])$, 
any compact set $\Omega\subset\G\bk \bH^3$,
and for any $\epsilon>0$, 
there is a constant $c=c(\Omega,\epsilon)$ such that
\begin{eqnarray}\label{eq:thm1b}
\sup_{z\in \Omega} \ |E_{\Gamma}(z,1+iT)|\ \leq \  cT^{1/2+\e} .    
\end{eqnarray}
As a consequence, we have
that 
\begin{eqnarray}\label{eq:nu2SL2Zi}
\nu_2(\SL_2(\Z[i])) = \nu_\infty(\SL_2(\Z[i])) =  1/2  .  
\end{eqnarray}
\end{thm}

We note that the upper bound for $\nu_\infty$ implied by \eqref{eq:thm1b} matches 
 a lower bound for $\nu_2$ in work of Kelmer-Yu \cite{KelmerYu2022} (see Remark \ref{Rmk:KY}), leading to \eqref{eq:nu2SL2Zi}.

\begin{remark}
We similarly expect that Theorem \ref{thm:1b} can be proved along the same lines for $\Gamma=\SL_2(\cO_K)$ 
(or
a congruence subgroup thereof)
with $K$ an imaginary quadratic field. 
When $\Gamma$ is a congruence lattice in $\SL_2(\C)$, the convexity sup norm bound is $\nu_\infty(\G)\leq 1$, it is commonly believed that
$\nu_\infty(\G)=\tfrac{1}{2}$, and the best currently known bound comes from the work of Assing \cite{Assing2019} who used amplification to show that $\nu_\infty(\G)\leq \tfrac{7}{8}$ (his result is more general and deals with Eisenstein series for $SL_2$ defined over general number fields). 
\end{remark}
\begin{remark}\label{Rmk:KY}    
For lattices acting on hyperbolic $(n+1)$-space,
even the ``convexity'' bound 
$\nu_\infty(\Gamma)\leq \tfrac{n}{2}$ 
is not known in general, not even for 
general arithmetic lattices.
Nevertheless, in \cite{KelmerYu2022},
this  convexity bound was proved  for a large family of arithmetic lattices. While it may be expected that $\nu_\infty=\tfrac{n-1}{2}$, there are no known subconvex bounds when $n\geq 3$.    
(Note that with a different normalization -- see Remark \ref{rmk:Enormalize} -- an argument due to Sarnak \cite{SarnakToMorawetz} shows that the $L^\infty$ 
norm is controlled by the $L^2$ norm on compacta, thus proving the convexity bound with this other normalization.)

For sup norm bounds on average, the convexity bound $\nu_2(\Gamma)\leq \tfrac{n}{2}$ is known to hold for a general non-uniform lattice \cite[Cor 7.7]{CohenSarnak1980}.  Note that when $n\in \{1,2,3,5,7\}$, for certain arithmetic lattices, it is possible to evaluate $E(s,z_0)$ at special points as a product of zeta functions and $L$-functions (see \cite{KelmerYu2022}), from which one can conclude that $\nu_2(\Gamma,z_0)=\tfrac{n-1}{2}$ at these points. This gives a lower bound on what can be expected to hold in general. 

\end{remark}

\subsection{Mean square bounds on Epstein zeta functions}\label{sec:Epstein}\ 

Both 
Theorems \ref{thm:1}
and \ref{thm:1b}
are derived (in section \ref{s:Eis}) from the following bounds on Epstein zeta functions. Given a positive-definite quadratic form $Q$ in $m$ variables, let
\begin{align*}
    Z_Q(s) := \sum_{v\in \Z^m \setminus 0} Q(v)^{-s} 
\end{align*}
denote the Epstein zeta function, defined originally in some half-plane $\Re(s)\gg1,$ and having a well-known meromorphic continuation to $s\in\C$ (see \cite{Terras1973}) and functional equation relating $Z_Q(s)$ to $Z_{Q_-}(\frac m2-s)$, where $Q_-$ is the dual form. That is, if $Q$ is given by $Q(x) = x^T Z x$, then $Q_-(x) = x^T Z^{-1}x$.

\begin{thm}\label{thm:Eps bounds}
    Let $Q$ be a positive-definite quadratic form  in $m$ variables, and let the associated Epstein zeta function be $Z_Q$. If $m=2$ then 
    \begin{align}\label{eps m=2}
        \int_T^{2T} |Z_Q(\frac{1}{2}+it)|^2 dt \ \ll_Q\  T \log^2(T).
    \end{align}
    Moreover, if $m \ge 3$, then for any $\varepsilon>0$,
    \begin{align}\label{e:Zeta m}
        \int_T^{2T}  \left|Z_Q\left(\frac {m}4+it\right)\right|^2 dt \  \ll_Q \ T^{m/2+\varepsilon}.
    \end{align}
    In either case, the implicit constants depend on the form $Q$, 
    but may be taken uniform as $Q$ varies in a compact set in the space of positive-definite quadratic forms.
\end{thm}

We note that the above result is only new for $m=2,3$.
While the $m=3$ case is not related to bounds on Eisenstein series, we nevertheless 
include it for its intrinsic interest.
For $m\ge 4$, \eqref{e:Zeta m} follows immediately from 
Blomer's pointwise bounds for the Epstein zeta function.
\begin{thm}[{\cite[Theorem 1]{Blomer2020}}]\label{thm:Blomer}
Let $Q$ be a positive-definite quadratic form  in $m$ variables, with $m\ge4$. Then for any $\vep>0$,
$$
\left|Z_Q\left(\frac m4 + i t\right)\right| \ll_{Q,\vep} T^{(m-2)/4+\vep}.
$$
\end{thm}
Blomer also gives 
pointwise bounds for $m=2$ and $m=3$, but these are weaker than what is needed for the $L^2$ bounds in Theorem \ref{thm:Eps bounds}.


\subsection{Values of indefinite forms at integer points}\ 

To prove Theorem \ref{thm:Eps bounds}, we require a result from the geometry of numbers, namely the following uniform version of
\cite[Theorem 2.3]{EskinMargulisMozes1998}.
\begin{thm}\label{thm:EMM}
For any $n=p+q\geq 3$ with $p\geq q\geq 1$, and any form $Q(v)$ of signature $(p,q)$ having discriminant one, there are constants $c=c(Q)$, $A_0$, and $B_0$ such that, for any $A\geq A_0$ and $B\geq B_0$, 
if $(p,q)\not\in \{(2,1),(2,2)\}$, then 
$$\#\{  v\in \Z^n: \|v\|\leq A,\; Q(v)\in (-B,B)\}\leq c BA^{n-2}.$$
while for $(p,q)\in\{(2,1),(2,2)\}$,
$$\#\{  v\in \Z^n: \|v\|\leq A,\; Q(v)\in (-B,B)\}\leq c BA^{n-2}\log(A).$$
The constant $c$ can be taken uniform for $Q$ ranging in a compact set.
\end{thm}
\begin{rem}
 The result of \cite[Theorem 2.3]{EskinMargulisMozes1998} gives a similar upper bound to $$\#\{  v\in \Z^n: \|v\|\leq A,\; Q(v)\in (a,b)\}$$ with the bound depending implicitly on the form $Q$ and the target interval $(a,b)$. The novelty of our result is to make the dependence on the target interval explicit.
\end{rem}
\subsection*{Outline}
In Section \ref{sec:EMM}, we prove Theorem \ref{thm:EMM}.
Then in Section \ref{s:m=2} we first focus on the case $m=2$ of Theorem \ref{thm:Eps bounds}; and settle the case $m\ge3$ in Section \ref{s:m=3}.  We then show in Section \ref{s:Eis} how to derive the bounds on the Eisenstein series from Theorem \ref{thm:Eps bounds}. Finally, in Section \ref{sec:PhilSar}, we explicate Remark \ref{rmk:PhilSar}.

\subsection*{Notation}
We use standard Vinogradov notation that $f\ll g$ if there is a constant $C>0$ so that $f(x)\le C g(x)$ for all $x$. When the implied constant depends on more than the lattice $\G$ or $Q$ or $z$ varying in a compact set $\Omega$, which we think of as fixed, we denote this with a subscript.

\subsection*{Acknowledgements}
We thank Valentin Blomer, Jens Marklof, and Amir Mohammadi for insightful discussions,
and the referee for many comments that greatly improved this paper.

\section{Counting estimate}\label{sec:EMM}
In this section we prove Theorem \ref{thm:EMM}. 
 We first recall the main ideas in the proof of  \cite{EskinMargulisMozes1998}. Let $Q$ be a quadratic form of signature $(p,q)$ and let $g\in \SL_n(\R)$ such that $Q(v)=Q_0(gv)$ with 
$$Q_0(x)=x_1x_n+\sum_{i=2}^p x_i^2-\sum_{i={p+1}}^{n-1}x_i^2.$$ 
We let $g$ (and hence $Q$) vary in a compact set, and fix a $\beta>0$ so that 
 $\max\{\|g\|, \|g^{-1}\|\}\leq \beta$ in this compact region. 

Let $H=\SO_{Q_0}(\R)$, let $K=H \cap \SO(n)$ be a maximal compact, and let 
$a_t=\diag(e^{-t},1\ldots,1,,e^t)\in H$.
For any real-valued, compactly supported function $f$ on $\R^n$, any $r>0$, and any $\xi\in \R$, define the function 
$$J_f(r,\xi)=\frac{1}{r^{n-2}}\int_{\R^{n-2}} f(r,x_2,\ldots, x_{n-1},\tfrac{\xi-Q_0(0,x_2,\ldots, x_{n-1},0)}{2r})dx_2\ldots dx_{n-1}.$$
Then by \cite[Lemma 3.6]{EskinMargulisMozes1998}, there is a constant $c_{p,q}$ and $T_0>1$ such that for every $t\geq \log(T_0)$ and any $v\in \R^n$ with $\|v\|>T_0$, we have that
$$\left|J_f(\|v\|e^{-t},Q_0(v))- c_{p,q}e^{(n-2)t}\int_K f(a_tkv)dm(k)\right|\leq 1.$$
In particular, if we choose $f$ so that $J_f(r,\xi)\geq 2$ for all $r\in (1,2)$ and $\xi\in (a,b)$, then for any $v\in \R^n$ with $e^t\leq \|gv\|\leq 2e^t$ and $Q_0(gv)\in (a,b)$, we have that 
$J_f(\|gv\|e^{-t}, Q_0(gv)))\geq 2$, whence
$c_{p,q}e^{(n-2)t}\int_K f(a_tkv)dm(k)\geq 1$ and
\begin{eqnarray}
\nonumber
\#\{v\in \Z^n: Q_0(gv)\in (a,b), \|gv\|\in [e^t,2e^t]\}&\leq& \sum_{v\in \Z^n} c_{p,q} e^{(n-2)t}\int_Kf(a_t kgv)dm(k)\\
\label{eq:fhat}
&=& c_{p,q}e^{(n-2)t}\int_K\hat{f}(a_t kg)dm(k)
\end{eqnarray}
where 
$$\hat{f}(g)=\sum_{0\neq v\in\Z^n}f(gv)$$ 
is the Siegel transform.
Next, using \cite[Lemma 2]{Schmidt1968} we can bound 
\begin{eqnarray}
\label{eq:alpha}
    \hat{f}(g)\leq c_f \alpha(g\Z^n),
\end{eqnarray}
where the function $\alpha(\Lambda)$ is the function on the space of lattices defined in \cite[equation (3.3)]{EskinMargulisMozes1998}, and
 $c_f$ is a constant depending only on $f$. Finally, \cite[Theorem 3.2 and Theorem 3.3]{EskinMargulisMozes1998} state that
 $$\int_K\alpha(a_tkg\Z^n )dm(k)\leq c_g,$$
 is uniformly bounded when $(p,q)\not\in\{ (2,2), (2,1)\}$ and that 
$$\int_K\alpha(a_tkg\Z^n)dm(k)\leq c_g t,$$
when $(p,q)=(2,2)$ or $(p,q)=(2,1)$.
Here the constant $c_g$ is uniform when $g$ is taken from a compact set. Using this we get that  when $(p,q)\not\in \{(2,2), (2,1)\}$
$$\#\{v\in \Z^n: Q_0(gv)\in (a,b), \|gv\|\in [e^t,2e^t]\}\leq c_{p,q}c_fc_g e^{(n-2)t},$$
 and that for $(p,q)\in \{(2,2), (2,1)\}$
$$\#\{v\in \Z^n: Q_0(gv)\in (a,b), \|gv\|\in [e^t,2e^t]\}\leq c_{p,q}c_fc_g te^{(n-2)t}.$$
Summing over $t\leq \log(T)$ in dyadic intervals gives a bound of the form 
$$\#\{v\in \Z^n: Q_0(gv)\in (a,b), \|gv\| \leq T\}\leq c T^{n-2},$$
when  signature $(p,q)\not\in \{(2,2),(2,1)\}$ and 
$$\#\{v\in \Z^n: Q_0(gv)\in (a,b), \|gv\| \leq T\}\leq cT^{n-2}\log(T),$$
otherwise. Here the constant $c$ depends on the signature on the group element $g$, and on the function $f$ (and hence on the interval $(a,b)$.

Up to this point, the proof is identical to the treatment in \cite{EskinMargulisMozes1998}. It is at this moment where we need one simple extra ingredient to make everything uniform, in the special case of a target interval $(-B,B)$. Our goal is to find a suitable function $f$ so that $J_f(r,\xi)\geq 2$ when $r\in [1,2]$ and $|\xi|\leq B$ such that $c_f\leq cB$.
We first note that \cite[Lemma 2]{Schmidt1968} implies that for any $R\geq 1$ and any lattice $\Lambda=g\Z^n$, we have the bound 
$$\#\{v\in \Lambda: \|v\|\leq R\}\leq cR^n\alpha(\Lambda),$$
with $c$ and absolute constant depending only on $n$. We combine this with the following simple observation.
\begin{lem}\label{lem:Lambda}
For any $x\in \R^n$ and any lattice $\Lambda$ we can bound 
$$\#\{v\in\Lambda: \|v-x\|\leq R\}\leq \#\{v\in \Lambda: \|v\|\leq 2R\}.$$
\end{lem}
\begin{proof}
If there is no $v\in \Lambda$ with $\|v-x\|\leq R$, then this is obvious. Otherwise, let $u\in \Lambda$ satisfy $\|u-x\|\leq R$; then $\|v-x\|\leq R$ implies that 
$$\|v-u\|=\|v-x+x-u\|\leq 2R.$$
Since $v\in \Lambda$ if and only if $v-u\in \Lambda$, this concludes the proof.
\end{proof}

We now describe our choice of $f$.  Assume that $B\geq (n-2)$ and let $f$ take values in $[0,2]$ and supported on $[0,3]\times [-2,2]^{n-2}\times [-2B,2B]$ such that $f(x)=2$ on $[1,2]\times [-1,1]^{n-2}\times [-B,B]$.  Note that for any $r\in [1,2]$ and $\xi\in [-B,B]$ and $(x_2,\ldots,x_{n-2})\in [-1,1]^{n-2}$ we have that $\tfrac{\xi-Q_0(0,x_2,\ldots, x_{n-1},0)}{2r}\in [-B,B]$ so that 
$$f(r,x_2,\ldots, x_{n-1},\tfrac{\xi-Q_0(0,x_2,\ldots, x_{n-1},0)}{2r})=2.$$
We thus get a lower bound 
$J_f(r,\xi)\geq \frac{2}{2^{n-2}}\int_{[-1,1]^{n-2}}dx_2\ldots dx_{n-1}=2 $.
On the other hand, we can cover the support of $f$ by $2B+1$ balls of radius $R_n=2\sqrt{n}$ centered at the points $v_j=(1,0,\ldots, 0,2j)$ with $-B\leq j\leq B$.
From Lemma \ref{lem:Lambda}, we can bound 
$$\#\{v\in \Lambda: \|v-v_j\|\leq R_n\}\leq \#(\Lambda\cap 2R_n\}\leq c_n\alpha(\Lambda),$$
where the constant depends only on $n$. Using this, we can bound the Siegel transform 
$\hat{f}(g)\leq 2c_nB \alpha(\Lambda)$.
Plugging this estimate into \eqref{eq:alpha} and \eqref{eq:fhat} concludes the proof of Theorem \ref{thm:EMM}.


\section{Bounds on the Epstein zeta function}
\subsection{Bounds for $m=2$} \label{s:m=2}
In this section we fix $m=2$ and prove \eqref{eps m=2}. In \cite{SavastruVarbanets2005}, the authors considered the case of an integral quadratic form $Q$ and proved an approximate functional equation as well as a formula for the mean square. While the formula for the mean square is special for a family of integral quadratic forms, the approximate functional equation \cite[Theorem 1]{SavastruVarbanets2005} holds in general. We record this result in our special case as follows.
\begin{thm}
For $Q$ a positive definite quadratic form of discriminant $D$ with dual form $Q_-$ and sequences $a_n$, $\lambda_n$, $b_n$, $\mu_n$ defined for $\Re(s)\gg1$ by:
$$
Z_Q(s) =
\sum_{\lambda_n} \frac{a_n}{\lambda_n^s}, 
\text{\qquad
and
\qquad}
Z_{Q_-}(s) =
\sum_{\mu_n} \frac{b_n}{\mu_n^s}, 
$$
we have, for $s=\tfrac{1}{2}+it$ with $|t|\geq 1$, that:
    $$Z_Q(s)=\sum_{\lambda_n\leq X}\frac{a_n}{\lambda_n^{s}}+\chi(s)\sum_{\mu_n\leq X}\frac{b_n}{\mu_n^{1-s}}+O_D(\log(|t|)),$$
where
$$
\chi(s)=(\frac{\sqrt{D}}{\pi})^{1-2s}\frac{\Gamma(1-s)}{\Gamma(s)},
$$
and
$X=X(t):= \frac{|t|\sqrt{D}}{\pi}$.
\end{thm}
Noting that $|\chi(s)|=1$ for $s=\tfrac{1}{2}+it$, we can estimate
\begin{eqnarray}\label{eq:ZQsq}
|Z_Q(\frac{1}{2}+it)|^2\ll 
F_Q(t)
+
F_{Q_-}(t)
+O(\log^2(t)),    
\end{eqnarray}
where we have set
$$
F_Q(t) :=
\left|\sum_{\lambda_n\leq X(t)}\frac{a_n}{\lambda_n^{\frac{1}{2}+it}}\right|^2.
$$
Recall again that $X(t)=\frac{|t|\sqrt{D}}{\pi}\asymp |t|$.

Integrating \eqref{eq:ZQsq} gives:
\begin{eqnarray}\label{eq:ZQbnd}
\int_{T}^{2T}|Z_Q(\tfrac{1}{2}+it)|^2dt&\ll&
\int_{T}^{2T} F_Q(t) dt+
\int_{T}^{2T} F_{Q_-}(t) dt+O(T\log^2(T)).
\end{eqnarray}
These terms can be estimated as:
\begin{eqnarray*}
\int_{T}^{2T} F_Q(t) dt &=& \int_{T}^{2T} \mathop{\sum_{u,v\in \Z^2\setminus0}}_{Q(u) ,Q(v) \leq X(t)} \frac{1}{ Q(u)^{1/2+it}Q(v)^{1/2-it}}dt \\
&\ll & \mathop{\sum_{u,v\in \Z^2\setminus0}}_{\|v\| ,\|u\| \ll \sqrt{T}} \frac{1}{ \|u\| \|v\|} \left|\int_{\max\{Q(u),Q(v),T\}}^{2T}e^{it\log(\frac{Q(v)}{Q(u)})}dt\right|,
\end{eqnarray*}
where we used that $Q(u)\asymp \|u\|^2$.

We now break this into different regions depending on the ratio of $\frac{Q(u)}{Q(v)}$. First consider the range when $|\log(\frac{Q(u)}{Q(v)})|\geq 1$. In this range, we can bound the inner integral by $2$ to get a bound of 
$$ \mathop{\sum_{u,v\in \Z^2}}_{\|v\| ,\|u\| \ll \sqrt{T}} \frac{1}{ \|u\| \|v\|} \ll  \left(\mathop{\sum_{v\in \Z^2}}_{\|v\|  \ll \sqrt{T}} \frac{1}{\|v\|} \right)^2\ll T.$$
For the rest, we have that $Q(u)\asymp Q(v)$ so also $\|u\|\asymp \|v\|$, and we break the sum into dyadic intervals
$A\leq \|v\|\leq 2A$ and $\frac{B}{T}\leq |\log(\frac{Q(u)}{Q(v)})|\leq \frac{2B}{T}$, with $A\leq \sqrt{T}$ and $B\leq T$.
The contribution of each such dyadic interval is then given by  
\begin{eqnarray*}
N_T(A,B)&=& \mathop{\sum_{v\in \Z^2}}_{A\leq\|v\|\leq 2A} \mathop{\sum_{u\in \Z^2}}_{ |\log(\frac{Q(u)}{Q(v)})|\in (\frac{B}{T},\frac{2B}{T})}\frac{1}{ \|u\| \|v\|} \left|\int_{\max\{Q(u),Q(v),T\}}^{2T}e^{it\log(\frac{Q(u)}{Q(v)})}dt\right|\\
&\ll& \frac{T}{A^2B}\#\{ u,v\in \Z^2: \|v\|\leq 2A,\; |\frac{Q(u)}{Q(v)}-1|\leq \frac{2B}{T}\}\\
 &\ll& \frac{T}{A^2B}\#\{ (u,v)\in \Z^4: \|(u,v)\|\leq 2A,\; |Q(u)-Q(v)|\ll \frac{A^2B}{T}\}.\\
\end{eqnarray*}

We also have the range where  $|\log(\frac{Q(u)}{Q(v)})|\leq \frac{1}{T}$ that we can similarly bound by 
\begin{eqnarray*}
N_T(A)&=& \mathop{\sum_{v\in \Z^2}}_{A\leq\|v\|\leq 2A} \mathop{\sum_{u\in \Z^2}}_{ |\log(\frac{Q(u)}{Q(v)})|\in [0,\frac{1}{T})}\frac{1}{ \|u\| \|v\|} \left|\int_{\max\{Q(u),Q(v),T\}}^{2T}e^{it\log(\frac{Q(u)}{Q(v)})}dt\right|\\
 &\ll& \frac{T}{A^2}\#\{ (u,v)\in \Z^4: \|(u,v)\|\leq 2A,\; |Q(u)-Q(v)|\ll 1\}.
\end{eqnarray*}

Note that for $Q(v)$ a positive definite binary quadratic form, the form 
$$\tilde{Q}(u,v)=Q(u)-Q(v)
$$ 
has signature $(2,2)$. 
Applying Theorem \ref{thm:EMM} gives
$$\#\{  v\in \Z^4: \|v\|\leq A,\; \tilde Q(v)\in (-B,B)\}\leq c BA^2\log(A),$$
from which  we can bound 
$$N_T(A,B)\ll A^2\log(A)\ll A^2\log(T)$$ 
for $A\leq \sqrt{T}$. 
We also have the bound
$N_T(A)\leq cT\log(T)$ for all $A\leq \sqrt{T}$.
Taking $A,B$ to be powers of $2$ and summing over $A\leq \sqrt{T}, B\leq T$, we get the bound 
$$\int_{T}^{2T}|Z_Q(\tfrac{1}{2}+it)|^2dt\ll T\log^2(T),$$
as claimed in \eqref{eps m=2}.

\subsection{Bounds for $m\ge 3$}\label{s:m=3}

Turning now to \eqref{e:Zeta m} for $m\ge 3$, we follow the same approach, however instead of using the approximate functional equation of \cite{SavastruVarbanets2005}, we instead use the following.
\begin{thm}[{\cite[(2.2)]{Blomer2020}}]
Recalling that $Q_-$ is the dual form to $Q=Q_+$,
we have  that
\begin{align}
    Z_Q(\frac m4+it) \ll 1+ |t|^\varepsilon \sum_{\pm} \sum_{A} \int_{|w|\le|t|^\varepsilon} A^{-m/4} \left| \sum_{v\neq 0} \frac{V_A(Q_{\pm}(v))}{Q_{\pm}(v)^{\pm i t + iw}} \right| dw
\end{align}
where the sum over $A$ ranges over powers of $2$ less than $|t|^{1+\varepsilon}$, and $V_A$ is bounded and has compact support in $[A,3A]$. 
\end{thm}

Now consider the $L^2$ norm and expand the square using Cauchy-Schwarz to get that
\begin{align*}
    &\int_{T}^{2T} |Z_Q(\frac m 4+it)|^2 d t \ll \int_{T}^{2T} \left|1+t^\epsilon \sum_{\pm} \sum_{A} \int_{|w|\le|t|^\varepsilon} A^{-m/4} 
    \left| \sum_{v\neq 0} \frac{V_A(Q_{\pm}(v))}{Q_{\pm}(v)^{\pm i t + iw}} \right| dw\right|^2 dt\\
    &\ll \int_{T}^{2T}\left(1+t^\epsilon \sum_{\pm}\sum_{A\leq t^{1+\epsilon}\atop\text{dyadic}}\int_{|w|\leq t^\epsilon} A^{-m/2}  \left| \sum_{v\neq 0} \frac{V_A(Q_{\pm}(v))}{Q_{\pm}(v)^{\pm i t + iw}} \right|^2 \right) dt\\
    &\ll T+T^\varepsilon  \sum_{\pm}\sum_{A\leq T^{1+\epsilon}\atop\text{dyadic}} A^{-m/2}\int_{|w|\leq T^{\epsilon}}  \int_{\max\{T,A,|w|^{1/\epsilon}\}}^{2T} \left| \sum_{v\neq 0} \frac{V_A(Q_{\pm}(v))}{Q_{\pm}(v)^{\pm i t + iw}} \right|^2  dt.
\end{align*}
For any fixed $A$ and $w$ in this range and for each of the choices of $Q=Q_+$ or $Q=Q_-$,
we open the square and  estimate the integral 
\begin{align*}
\int_{\max\{T,A,|w|^{1/\epsilon}\}}^{2T} \left| \sum_{v\neq 0} \frac{V_A(Q(v))}{Q(v)^{ i t + iw}} \right|^2  dt
\ll  
\sum_{v,u\neq 0} V_A(Q(v))V_A(Q(u)) 
\left|
\int_{\max\{T,A,|w|^{1/\epsilon}\}-w}^{2T-w}e^{it\log\frac{Q(v)}{Q(u)}}  dt
\right|
\end{align*}
depending on the range of $\frac{Q(v)}{Q(u)}$.


If $|\log(\frac{Q(v)}{Q(u)})| \ge 1$, the integral is bounded independently on the boundaries of integration. Since $V_A(Q(v))$ is bounded and supported on $Q(u)\leq 3A$, the sum over 
$\|u\| \ll Q(u)^{1/2} \ll A^{1/2}$ and $\|v\| \ll A^{1/2}$ is bounded by $O(A^m)$.
In the remaining case we may assume that $\| v\| \asymp \| u\| \asymp A^{1/2}$. Now further break the summation by taking $\frac{B}{T} \le |\log(\frac{Q(v)}{Q(u)})| \le \frac{2B}{T}$, with $B\ll T$ dyadic. For each $B\geq 1$ we get a term of the form 
\begin{align*}
    N_T(A,B) &= \sum_{\substack{v \in \Z^m \\ \sqrt{A} \le \|v\| \le \sqrt{3A}  }} \sum_{\substack{u \in \Z^m \\ |\log(\frac{Q(v)}{Q(u)})|\in (\frac{B}{T},\frac{2B}{T})    }}\left|  \int_{\max\{T,A,|w|^{1/\epsilon}\}-w}^{2T-w} e^{i t\log(\frac{Q(v)}{Q(u)})}    dt\right|\\
    & \ll \frac{T}{B}  \# \{ (u,v) \in \Z^{2m}  :  \|(u,v)\| \le \sqrt{2 A}, \ |Q(u) - Q(v)| \ll \frac{A B}{T} \}
\end{align*}
and for $B\leq 1$ we let
\begin{align*}
    N_T(A) &= \sum_{\substack{v \in \Z^m \\ \sqrt{A} \le \|v\| \le \sqrt{3A}  }} \sum_{\substack{u \in \Z^m \\ |\log(\frac{Q(v)}{Q(u)})|\in (0,\frac{1}{T})    }}\left| \int_{\max\{T,A,|w|^{1/\epsilon}\}-w}^{2T-w} e^{i t \log(\frac{Q(v)}{Q(u)})}    dt\right|\\
    & \ll T  \# \{ (u,v) \in \Z^{2m}  :  \|(u,v)\| \le \sqrt{2 A}, \ |Q(u) - Q(v)| \ll 1 \}.
\end{align*}
Again following our approach for $m=2$, we note that $\tilde{Q}(u,v):=Q(u)-Q(v)$ is a form of signature $(m,m)$, and apply Theorem \ref{thm:EMM}. It follows that 
\begin{equation}
  N_T(A,B) \ll \frac{T}{B}  A^{m-1} \frac{A B}{T} = A^{m}.
\end{equation}
Similarly, we can trivially bound $N_T(A)$ by $A^m$. 

Plugging these bounds back, bounding the integral of $w$ by $O(T^{\epsilon})$ and summing over $A,B$ powers of $2$ we get the bound
\begin{equation*}
    \int_{T}^{2T} |Z_Q(1+it)|^2 d t \ll T+ T^\varepsilon \sum_{A\leq T^{1+\epsilon}\atop\text{dyadic}}
    A^{-m/2}\left(A^m+N_T(A)+\sum_{B\ll T\atop\text{dyadic}} N_T(A,B)  \right)\ll T^{m/2+\varepsilon},
\end{equation*}
as claimed in \eqref{e:Zeta m}.

\section{Bounds on  Eisenstein series: Proof of Theorems \ref{thm:1} and \ref{thm:1b}} \label{s:Eis}




To move from bounds on the Epstein zeta function to the Eisenstein series, 
argue as follows.
First note that 
for $\Gamma=\SL_2(\Z)$ we have the following well known identity (for $\Re s>1$):
$$\zeta(2s)E_\Gamma(s,z)=y^s\sum_{(c,d)\neq (0,0)}  \frac{1}{((x^2+y^2)c^2+2xcd+d^2)^{s}} =y^sZ_Q(s),$$
where $Q(c,d)=Q_z(c,d)=(x^2+y^2)c^2+2xcd+d^2$ and
$$Z_Q(s)=\sum_{v\in \Z^2\setminus 0} Q(v)^{-s}.$$ 
Applying \eqref{eps m=2}
and the well-known bound $\zeta(1+2it)\gg \log(t)^{-1}$, we conclude \eqref{eq:thm1a}.
This completes the proof of Theorem \ref{thm:1}.
\\

Next we show that for $\Gamma=\SL_2(\Z[i])$ we can again write $E_\Gamma(s,z)$ in terms of an Epstein zeta function. Using the  upper-half-space model of hyperbolic $3$-space, 
$$
\bH^3=\{z= x_1 + ix_2 + j y \ : \ x_j\in\R, y>0\},
$$
the Eisenstein series is defined (for $\Re s>2$) by
$$E_\Gamma(s,z)=\sum_{\substack{c,d \in \Z[i]\\ \text{co-prime}}}\frac{y^s}{N(cz+d)^{s}},$$
where 
$N(z)=x_1^2+x_2^2+y^2$ denotes the norm on the quaternions. Since the norm is multiplicative, if we write
\begin{align*}
  \zeta_{\Z[i]}(s)=\sum_{\alpha \in \Z[i]} \frac{1}{N(\alpha)^{s}},
\end{align*}
then we can simplify the Eisenstein series to
\begin{align*}
  \frac{1}{4}\zeta_{\Z[i]}(s)E(s,z) = \sum_{\substack{(c,d)\neq (0,0)}}\frac{y^s}{N(cz+d)^{s}}.
\end{align*}
Moreover, we can write $\zeta_{\Z[i]}(s) = 4 \zeta(s) L(s,\chi_1)$ with $\chi_1$ the quadratic Dirichlet character modulo $4$. Now if we expand the norm on the right hand side, we arrive at
\begin{align*}
  \zeta(s) L(s,\chi_1)E(s,z) = y^s \sum_{\substack{(c,d)\neq (0,0)}}\frac{1}{Q_z(c,d)^{s}},
\end{align*}
where
\begin{align*}
  Q_z(c,d) &= N(z) c_1^2 + N(z) c_2^2+d_1^2+d_2^2+2(x_1c_1d_1-x_2c_2d_1+x_1c_2d_2 +  x_2 c_1d_2)
\end{align*}
is a positive definite quaternary quadratic form.
Again we have good control on $\zeta(s)$ for $s=1+it$ and on $L(s,\chi_1)$. Thus the problem reduces to estimating the Epstein zeta function
\begin{align*}
  Z_{Q_z}(s)=\sum_{v\in \Z^4\setminus 0} Q_z(v)^{-s},
\end{align*}
and \eqref{eq:thm1b} follows easily from Theorem \ref{thm:Blomer}.
This completes the proof of Theorem \ref{thm:1b}.

\section{Sharpness}\label{sec:PhilSar}
We now consider the case of a general non-arithmetic lattice and show that any subconvex bound for $\nu_2(\Gamma)$ implies the existance of infinitely many Maass cusp forms. Explicitly we show the following.
\begin{thm}\label{thm:PhilSar}
Let  $\Gamma\leq \PSL_2(\R)$ be a non uniform lattice and assume that there is some $z_0$ such that  $\nu_2(\Gamma,z_0)<1/2$. Then there are infinitely many Maass cusp forms  $\vf_j\in L^2(\Gamma \bk \bH)$  with $\triangle\vf_j+\lambda_j\vf_j=0$. Moreover, they satisfy the local Weyl law: for any test function $h(r)$ with Fourier transform smooth and compactly supported, for all sufficiently large $T$,
\begin{eqnarray*}
\sum_{j} h(\frac{r_j}{T})|\vf_j(z)|^2
&=&T^2 \frac{|\Gamma_z|}{2\pi}\int_0^\infty h(r)rdr +O_{\delta,h}(T^{2-\delta})\\
\end{eqnarray*}
    for some $\delta>0$, where $\Gamma_z=\{\gamma\in \Gamma: \gamma z=z\}$ and $\{\vf_j\}_{j\in \N}$ form
    an orthonormal system of 
    Maass forms in $L^2(\G\bk \bH)$ with eigenvalue parametrized by $\lambda_j=\tfrac{1}{4}+r_j^2$.
\end{thm}
\begin{proof}We recall some well known results on the pre-trace formula and refer to \cite{Hejhal1976} for more details.
Given a point pair invariant $k(z,w)=k(\sinh^2(d(z,w)))$ with $d(z,w)$ the hyperbolic distance and $k\in C^\infty_c(\R^+)$, its spherical transform is defined as  $H(s)=\int_{\bH^2} k(z,i)\Im(z)^sd\mu(z)$. By \cite[Proposition 4.1]{Hejhal1976} the point pair invariant can be recovered from $H(s)$ as follows : Let $h(r)=H(\tfrac{1}{2}+ir)$ and let $g(u)=\frac{1}{2\pi}\int_{-\infty}^\infty h(r)e^{-iru}dr$ denote its Fourier transform, then, defining the auxiliary function $Q\in C^\infty_c(\R^+)$ by $g(u)=Q(\sinh^2(\frac{u}{2}))$ we have that $k(t)=-\frac{1}{\pi}\int_t^\infty \frac{dQ(r)}{\sqrt{r-t}}$. We also recall that 
$k(0)=\frac{1}{2\pi}\int_0^\infty h(r)r\tanh(\pi r)dr$ (see \cite[Proposition 6.4]{Hejhal1976}).

Given any such point pair invariant we have the pre-trace formula 
$$\sum_{\gamma\in\Gamma} k(z,\gamma z)=\sum_{j} h(r_j)|\vf_j(z)|^2+\sum_{i=1}^\kappa\frac{1}{2\pi}\int_\R h(r)| E_{\Gamma,\xi_i}(\tfrac{1}{2}+ir,z)|^2dr,$$
where $\xi_1,\ldots,\xi_\kappa$ are the cusps of $\Gamma$. 

Now, fix a smooth compactly supported function $g(u)\in C^\infty_c((-1,1))$ and for any $T\geq 1$ let $g_T(u)=Tg(Tu)$ so that $h_T(r)=h(\frac{r}{T})$ and $k_T(z,w)$ the corresponding point pair invariant. Since $g_T(u)$ is supported on $(-\frac{1}{T},\frac{1}{T})$ the point pair invariant $k(z,w)$ is supported on the set 
$\{(z,w)| d(z,w)\leq \frac{1}{T}\}$ with $d(z,w)$ the hyperbolic distance. Since $\Gamma$ acts properly discontinuously on $\bH^2$ for any fixed $z$ there is $\delta=\delta(z)$ such that  $d(z,\gamma z)\geq \delta$ for any $\gamma\in \Gamma$ with $\gamma z\neq z$. In particular taking $T_0\geq \frac{1}{\delta(z)}$ for any $T\geq T_0$ we have that 
$k_T(z,\gamma z)=0$ if $\gamma z\neq z$. Hence for any $T\geq T_0$ we have 
\begin{eqnarray*}
\sum_{j} h(\frac{r_j}{T})|\vf_j(z)|^2+\sum_{i=1}^\kappa\frac{1}{2\pi}\int_\R h(\frac{r}{T})| E_{\Gamma,\xi_i}(\tfrac{1}{2}+ir,z)|^2dr&=& |\Gamma_z|k(0).\\
\end{eqnarray*}

Denote by $\nu_2=\nu_2(\Gamma,z)$ and note that for any $\ell\geq 0$ we can bound 
$h(t)\ll_{\ell,h} |r|^{-\ell}$. 
We can thus bound the contribution of the integrals over Eisenstein series by 
\begin{eqnarray*}\frac{1}{2\pi}\int_\R h(\frac{r}{T})| E_{\Gamma,\xi_i}(\tfrac{1}{2}+ir,z_0)|^2dr&=& \sum_{k\in \Z} \int_{kT}^{(k+1)T}h(\frac{r}{T})| E_{\Gamma,\xi_i}(\tfrac{1}{2}+ir,z)|^2dr\\
\ll_h T^{2\nu_2+1}.
\end{eqnarray*}
On the other hand, since the right hand side is
$$|\Gamma_z| k_T(0)= \frac{|\Gamma_z|}{2\pi}\int_0^\infty h(\frac{r}{T})r\tanh(\pi r)dr= T^2\frac{|\Gamma_z|}{2\pi}\int_0^\infty h(r)rdr+O_h(1),$$
we can conclude that if $\nu_2<\tfrac12$, then for any $\delta\in (0,1-2\nu_2)$ and any $T\geq T_0$ 
\begin{eqnarray*}
\sum_{j} h(\frac{r_j}{T})|\vf_j(z)|^2
&=&T^2 \frac{|\Gamma_z|}{2\pi}\int_0^\infty h(r)rdr +O_{\delta,h}(T^{2-\delta}).\\
\end{eqnarray*}
This completes the proof.

\end{proof}

Reiterating Remark \ref{rmk:PhilSar}, if one believes the Phillips-Sarnak conjecture \cite{PhillipsSarnak1985}, then the convexity $L^2$ bounds on Eisenstein series should be sharp for generic lattices. So  \eqref{eq:thm1a} really is relying heavily on the arithmeticity of $\SL_2(\Z)$.

\bibliographystyle{alpha}

\bibliography{DKbibliog}

\end{document}